\documentclass[a4paper,12pt,reqno,draft]{amsart}
\usepackage{amssymb,amsmath,array,amscd,amsthm,hhline}

\usepackage[mathscr]{euscript}
\usepackage{stmaryrd}
\usepackage{ulem}

\usepackage{tikz-cd}

\usepackage{color}

\usepackage{mathrsfs}

\renewcommand{\labelenumi}{{\rm \theenumi}}
\renewcommand{\theenumi}{{\rm(\arabic{enumi})}}
\renewcommand\epsilon{\varepsilon}
\renewcommand\emptyset{\varnothing}

\def\lm{\lambda}
\def\Lm{\Lambda}
\renewcommand\kappa{\varkappa}

\def\<{\langle}
\def\>{\rangle}

\def\ito{\stackrel\sim\to}

\def\j{\mathbf j}

\def\C{\mathbb C}
\def\DD{\mathbb D}

\def\Z{\mathbb Z}

\def\D{\mathcal D}

\newtheoremstyle{mytheoremstyle} 
    {\topsep}                    
    {\topsep}                    
    {\rmfamily}                   
    {1em}                           
    {\bf}                   
    {.}                          
    {.5em}                       
    {}  

\newtheorem{theorem}{Theorem}
\newtheorem{lemma}{Lemma}
\newtheorem{corollary}{Corollary}
\newtheorem{remark}{Remark}

\def\le{\leqslant}
\def\ge{\geqslant}

\def\pt{{\rm pt}}
\def\Loc{{\rm Loc_A}}
\def\loc{{\rm Loc}}

\def\IC{\mathop{\rm\mathbf{IC}}\nolimits}

\def\rad{\mathop{\rm rad}}
\def\id{\mathop{\rm id}}
\def\im{\mathop{\rm im}}

\def\GL{\mathop{\rm GL}}
\def\SL{\mathop{\rm SL}}

\def\suchthat{\mathbin{\rm |}}
\def\and{\,\mathbin{\&}\,}

\def\Hom{\mathop{\rm Hom}\nolimits}

\def\pre{{\rm pre}}

\def\ad{{\rm ad}}
\def\Ad{{\rm Ad}}

\def\Lie{\mathop{\mathrm{Lie}}}
\def\Ball{\mathrm K}

\def\h{\mathfrak h}

\def\b{\mathfrak b}

\def\X{\mathfrak X}

\def\S{\mathcal S}

\def\O{\mathcal O}

\def\N{\mathcal N}
\def\M{\mathscr M}

\def\sh{{\rm sh}}

\def\B{\mathscr B}
\def\BB{\mathcal B}
\def\C{\mathbb C}

\def\L{{\mathcal L}}

\def\={\equiv}

\renewcommand{\(}{\left(}
\renewcommand{\)}{\right)}

\def\sectsign{\mathhexbox278}

\def\BB{\mathcal B}

\def\csh#1#2{{\uwave {#1}}{}_{{}_{\scriptstyle #2}}}

\def\NN{\widetilde{\mathcal N}}
\def\Nc{\mathcal N}
\def\Hp{{\vphantom{H^H}}^p\!}

\def\g{\mathfrak g}

\newcommand\eqto{\xrightarrow{
   \,\smash{\raisebox{-0.65ex}{\ensuremath{\sim}}}\,}}

\title{Composition factors for the Springer resolution}
\author{Vladimir Shchigolev}

\begin{document}

\maketitle

\begin{abstract}
Let $\pi:\NN\to\Nc$ be the Springer resolution of the nilpotent cone for a semisimple connected algebraic group $G$ over $\C$
and $k$ be an arbitrary field.
What happens to $\pi_*\csh k\NN[\dim\NN]$ if the decomposition theorem
fails for it? We show that in this case, some additional (with respect to the case $\mathop{\rm char}k$=0) composition factors
of this direct image in the (abelian) category of perverse sheaves may emerge.
These factors emerge from the $Z_G(x)/Z_G(x)^0$-composition factors of the radicals of certain intersection forms and
from that of the top comohologies of Springer fibres (in the non-semisimple case).
\end{abstract}

\section{Introduction}  The recent Hodge theoretic proof of the decomposition theorem by de Cataldo and
Migliorini~\cite{dCM02},~\cite{dCM05} relies on the non-degeneracy of some intersection forms. An extension of this
argument~\cite[Theorem 3.7]{psh} allowed the authors to prove that over a filed of characteristic $p>0$, the decomposition theorem
for semismall morphisms holds if the $p$-modular reductions of these forms are non-degenerate and certain local systems are semisimple.

The aim of this paper is to study the Springer resolution $\pi:\NN\to\Nc$ as an example of such a semismall morphism in the
case when the decomposition theorem fails for $\pi_*\csh k\NN[\dim\NN]$, where $k$ is a field. This direct image is a
perverse sheaf regardless of the characteristic of $k$. So we can study its composition factors in the the abelian category
of perverse sheaves. 

The paper is organized as follows. In Section~\ref{Notation_and_Setup}, we recall the main constructions
connected with Springer resolutions and Slodowy slices. In Section~\ref{Neighbourhood_bases}, we construct a neighbourhood
basis $\{U_x(t)\}$ for a point $x$ in the nilpotent cone satisfying certain properties.
As we plan to consider $kZ_G(x)/Z_G(x)^0$-modules, special attention is paid to the equivariancy of all our constructions here
with respect to a reductive centralizer of $Z_G(x)$. This aim is mainly achieved with the help of Lemma~\ref{lemma:0}
and Corollary~\ref{corollary:0}, which are eqivariant versions of constructions from~\cite{CG97}.

The main result of this paper is Theorem~\ref{theorem:1} proved in Section~\ref{Composition multiplicities_and_semisimplicity},
which is an inequality relating composition factors in the category of perverse sheaves (left-had side) and the
composition factors in the category of $kZ_G(x)/Z_G(x)^0$-modules (right-hand side).

Our calculations here rely on the cohomologies of the compliments to the Springer fibers in the resolutions of the corresponding
Slodowy slices. Note that the cohomologies of similar (smooth) spaces, which may depend substantially on the characteristic
of the filed of coefficients, were already considered in~\cite{perv}, where they were used in Deligne's formula
for the intermediate extension. Although we do not use this formula, we encounter similar cohomologies,
when we calculate the composition of the form $i^*R\xi_*$ (see formula~(\ref{eq:c})). As a result, we get some additional
composition factors of the direct image resulting from the degeneracy of the intersection forms (described in
Section~\ref{Intersection_form_and_cohomology_of_difference}). Their structure is studied in Section~\ref{structureiM}.
More composition factors may come from the non-semisimplicity of the
top cohomology modules for the Springer resolution. On the other hand, if all these forms are non-degenerate and
the top cohomology modules are semisimple, then the decomposition theorem holds,
as also follows from Theorem~\ref{theorem:1}.


\newpage

\section{Notation and setup}\label{Notation_and_Setup}

\subsection{Springer resolution of nilpotent cone}
\label{srnc} Let $G$ be a semisimple connected algebraic group over $\C$, $\g$ be the Lie algebra of $G$ and $\Ad:G\to\GL(\g)$
be the adjoint representation. We make $G$ act on the right on $\g$ by $g\cdot u=\Ad(g)(u)$ for $g\in G$ and $u\in\g$.
We shall consider $\g$ in the metric topology. If $G$ acts on a set $X$ and $g\in G$, then $m_g:X\to X$ denotes
the action by $g$.

The set $\mathcal N$ of all nilpotent elements $x\in\g$ is called the {\it nilpotent cone} of $\g$.
This set is stable under the above action of $G$ and hence decomposes into the union of orbits.
We arbitrarily choose one point in any such orbit and denote by $\Lambda$ the set of all chosen points.
Let $\O_x$ denote the orbit containing $x\in\Nc$. Thus we have the stratification $\X=\{\O_x\}_{x\in\Lm}$,
which is known to be a Whitney stratification (see~\cite{CG97}).

We introduce the following partial order on $\Lambda$: $y\le x$ if and only if $\O_y\subset\overline{\O_x}$.
This order suggests the following terminology: a subset $\Phi\subset\Lm$ is called {\it closed} if $y\le x$ and $x\in \Phi$
imply $y\in \Phi$ and is called {\it open} if its compliment $\overline\Phi=\Lm\setminus \Phi$ is closed.
Equivalently, a subset $\Phi\subset\Lm$ is open if and only if $y\ge x$ and $x\in \Phi$ imply $y\in \Phi$.
Thus the set $\Nc_\Phi:=\bigsqcup_{x\in \Phi}\O_x$ is closed (resp. open) if $\Phi$ is closed (resp. open).

We briefly describe the Springer resolution for $\Nc$.
Let $\BB$ denote the variety of all Borel subalgebras of $\g$.
Let
$$
\NN:=\{(x,\b)\in\Nc\times\BB\,\suchthat\,x\in\b\}.
$$
The projection $\pi:\NN\to\Nc$ to the first component is called the {\it Springer resolution}.
On the other hand, the projection $\NN\to\BB$ to the second component identifies $\NN$ with the cotangent bundle $T^*\BB$.
It is well known that $\pi$ is semismall and proper.

For any subset $\Phi\subset\Lambda$, we set $\NN_\Phi:=\pi^{-1}(\Nc_\Phi)$ and $\pi_\Phi:=\pi\big|_{\NN_\Phi}:\NN_\Phi\to\Nc_\Phi$.
Similarly to $\N_\Phi$, the subset $\NN_\Phi$ is closed (resp. open) if $\Phi$ closed (resp. open).
Clearly, $\pi_\Phi$ is proper, as $\NN_\Phi$ is the full preimage.
We denote by $i_\Phi:\Nc_\Phi\to\Nc$ and $\tilde\imath_\Phi:\NN_\Phi\to\NN$ the natural inclusions.
The preimage $\B_x:=\pi^{-1}(x)$ of an element $x\in\Nc$ is called the {\it Springer fibre}.
For different points $x$ in the same orbit, Springer fibres $\B_x$ are isomorphic, as it follows
from $g\B_x=\B_{gx}$.

We shall write $\dim X$ for the complex dimension of X (vector space, variety).
We set
$$
d_\Nc:=\dim\Nc=\dim\NN,\;\;\; b_x:=\dim\B_x\;\;\text{ and }\;\;d_x:=\dim\O_x\;\;\text{ for }\;\;x\in\Lambda.
$$
Recall that $d_\Nc-d_x=2b_x$. The universal coefficient theorem and the main result of~\cite{DLP} show that $H^m(\B_x,k)=0$
for odd $m$. We use these facts throughout the paper without mention.

\subsection{Slodowy slice}\label{Slodowy_slice} We describe the
main constructions of~\cite{Slodowy_Four_Lectures} that we need here.
Take any $x\in\Lm$. By the Jacobson-Morozov lemma,
there exists a Lie algebra homomorphism $\omega:\mathfrak{sl}_2(\C)\to\g$ mapping the matrix
$\Big(\arraycolsep=2pt\begin{array}{cc}0&1\\[-2pt]0&0\end{array}\Big)$ to $x$.
Let $y$ be the image of $\Big(\arraycolsep=2pt\begin{array}{cc}0&0\\[-2pt]1&0\end{array}\Big)$
and $h$ be that of $\Big(\arraycolsep=2pt\begin{array}{cc}1&\,\;0\\[-2pt]0&\!{-}1\end{array}\Big)$.
These elements satisfy the relations $[x,y]=h$, $[h,y]=-2y$, $[h,x]=2x$.

The homomorphism $\omega$ defines the representation $\widetilde\omega:\SL_2(\C)\to\GL(\g)$ by exponentiating. The image of
$\widetilde\omega$ consists of inner automorphisms of $\g$. Take an arbitrary $t\in\C^*$ and consider the following Lie algebra
automorphisms:
$$
{\rho}(t):=\widetilde\omega\Big(\Big(\arraycolsep=2pt\begin{array}{cc}t&0\\[-2pt]0&t^{-1}\end{array}\Big)\Big),\qquad \sigma(t)(z)=tz.
$$
They clearly preserve the Jordan types of elements of $\g$.

The {\it Slodowy slice} is the following affine subspace of $\g$ (see~\cite[\sectsign 2.4]{Slodowy_Four_Lectures}):
$$
S_x:=x+\ker\ad\,y.
$$
As explained in~\cite[\sectsign 2.5]{Slodowy_Four_Lectures}, the automorphism $\j(t):=\sigma(t^2)\rho(t^{-1})$
stabilizes $S_x$ as well as each $S_x\cap\O_y$.
We shall use the following notation $t*z=\j(t)(z)$ for brevity. Clearly, $\lim_{t\mapsto0}t*z=x$ for any $z\in S_x$.
This property and the decomposition of the tangent space $T_x(\g)=T_x(\O_x)\oplus T_x(S_x)$ show that $S_x\cap\O_x=\{x\}$ and
$S_x\cap\O_y=\emptyset$ if $y\not\ge x$.

Let $\h$ be the Cartan subalgebra of $\g$ and $W$ the corresponding Weyl group.
In~\cite[3.4]{Slodowy_Four_Lectures}, the following simultaneous resolution was considered:
$$
\begin{CD}
\widetilde S_x@>\psi_x>> S_x\\
@V\theta_xVV @VV\chi_{\vphantom{A^A}x}V\\
\h@>\phi>>\h/W
\end{CD}
$$
Here $\widetilde S_x=\{(z,\b)\in S_x\times\BB\,\suchthat\,z\in\b\}$ and $\psi_x$ erases the second component.
We shall use the following notations for preimages: $\widetilde S_{x,h}=\theta^{-1}_x(h)$ and
$\widetilde S_{x,\bar h}=\chi^{-1}_{\vphantom{A^A}x}(\bar h)$ for $h\in\h$ and $\bar h\in\h/W$.
Thus $S_{x,0}=S_x\cap \N$, $\widetilde S_{x,0}=\pi^{-1}(S_x\cap\Nc)$.
In the special case $x=0$, we have $S_x=\g$ and $\psi_x$ restricts to the Springer resolution $\pi$.

One can easily extend the $*$-action of $\C^*$ from $S_x$ to $\widetilde S_x$ by $t*(z,\b)=(t*z,\j(t)(\b))$.
Thus the map $\psi_x:\widetilde S_x\to S_x$ becomes $*$-equivariant. In what follows, we shall not use the action
of the whole group $\C^*$ but only of its subgroup $(0,+\infty)$.

\subsection{Local systems}\label{local_systems} It is well known that $k$-local systems on a stratum $\O_x$ are given by $k\Pi_1(\O_x)$-modules.
We denote the corresponding functor by $\loc$. So we have $\L\cong\loc(\L_x)$ for any local system $\L$ on $\O_x$.

Consider the component group $A(x)=Z_G(x)/Z_G(x)^0$.
It is useful to define the functor $\Loc$ from the category of $kA(x)$-modules to the category of $k$-local systems on $\O_x$
as follows. Consider the fibration $Z_G(x)\to G\twoheadrightarrow^{\hspace{-10pt}\beta\vphantom{\beta_{\beta_\beta}}}\;\O_x$,
where $\beta(g)=g\cdot x$. Then we have the following exact sequence of groups:
$$
\Pi_1(G)\to\Pi_1(\O_x)\twoheadrightarrow^{\hspace{-10pt}\alpha\vphantom{\beta_{\beta_\beta}}}\;\Pi_0(Z_G(x))=A(x).
$$
Any $kA(x)$-module $M$ can be considered as a $k\Pi_1(\O_x)$-module $M^\alpha$ by $gm=\alpha(g)m$ for $g\in\Pi_1(\O_x)$
and $m\in M$. We set $\Loc(M):=\loc(M^\alpha)$. As $\alpha$ is epimorphic, $\Loc(M)$ is irreducible if $M$ is so.

We also describe the map $\alpha$ for further use. Let $\gamma:[0,1]\to\O_x$ be a loop based at $x$.
It can be lifted to a curve $\delta:[0,1]\to G$. We can assume that $\delta(0)=\mathbf{1}_G$ but can not in general
guarantee that $\delta$ is a loop. So the map $\alpha$ just describes this defect
taking the image of $\gamma$ in $\Pi_1(\O_x)$ to $\delta(1)Z_G(x)^0$.

\subsection{Intersection form and cohomology of difference}\label{Intersection_form_and_cohomology_of_difference} 
We shall use the following interpretation of the intersection form
for Springer fibers. Recall from~\cite[4.3]{Slodowy_Four_Lectures} that $\B_x$
is a deformation retract of $\widetilde S_{x,0}$. Therefore the map $H^{2b_x}(\widetilde S_{x,0},k)\to H^{2b_x}(\B_x,k)$
induced by the natural inclusion $\iota:\B_x\to\widetilde S_{x,0}$ is an isomorphism. Hence
we get the following sequence of $k$-linear maps:
\begin{equation}\label{eq:b}
\begin{CD}
\!H^{2b_x}(\B_x,k)^\vee\!\!@>\sim>\text{\tiny\sf induced by }\iota>\!\!H^{2b_x}(\widetilde S_{x,0},k)^\vee\!\!@>\sim>\text{\tiny\sf Poincar\'e duality}>\!\!H_c^{2b_x}(\widetilde S_{x,0},k)\!\!@>>\text{\tiny\sf induced by }\iota>\!\!H^{2b_x}(\B_x,k).
\end{CD}
\end{equation}
The composition defines a bilinear product $(\cdot,\cdot)_x:H^{2b_x}(\B_x,k)^\vee\times H^{2b_x}(\B_x,k)^\vee\to k$.
Note that although $\widetilde S_{x,0}$ depends on the chose of $y$ in the triple $(x,h,y)$,
the product $(\cdot,\cdot)_x$ does not.

As the map $\pi$ is $G$-invariant, the centralizer $Z_G(x)$ stabilizes $\B_x$ and thus acts on $H^\bullet(\B_x,k)$ by
$g\cdot h=m_{g^{-1}}^*(h)$. Moreover, its connected component $Z_G(x)^0$
acts identically on the cohomologies $H^\bullet(\B_x,k)$. 
So the component group $A(x)$ also acts on $H^\bullet(\B_x,k)$. Hence it also acts on
$H^{2b_x}(\B_x,k)^\vee$ by $g\cdot f(u)=f(g^{-1}\cdot u)$. It is easy to see that the above product $(\cdot,\cdot)_x$
is $A(x)$-invariant. 
Therefore the radical $\rad\big(H^{2b_x}(\B_x,k)^\vee\big)$ of this bilinear form is a $kA(x)$-submodule.


Now we introduce the following action of $A(x)$ on $\Hom_c^{2b_x}(\widetilde S_{x,0}\setminus\B_x,k)$.
Recall~\cite[5.1--5.4]{E} that $A(x)\cong C(x,y)/C(x,y)^0$, where $C(x,y)=Z_G(x)\cap Z_G(y)$.
As $C(x,y)$ stabilizes $\widetilde S_{x,0}\setminus\B_x$, the group $C(x,y)$ acts on
$\Hom_c^{2b_x}(\widetilde S_{x,0}\setminus\B_x,k)$ so that its subgroup $C(x,y)^0$ acts identically. 
Hence $A(x)$ also acts on $\Hom_c^{2b_x}(\widetilde S_{x,0}\setminus\B_x,k)$. A similar argument shows how to define
the action of $A(x)$ on $\Hom_c^{2b_x}(\widetilde S_{x,0},k)$.

The triviality of the odd degree cohomologies of Springer fibres gives the following exact sequence (see~\cite[III.7.6]{Iversen}):
$$
0\to H^{2b_x}_c(\widetilde S_{x,0}\setminus\B_x,k)\to  H^{2b_x}_c(\widetilde S_{x,0},k)\to  H^{2b_x}(\B_x,k)\to H^{2b_x+1}_c(\widetilde S_{x,0}\setminus\B_x,k)\to0.
$$
Clearly, all maps in this sequence are $A(x)$-invariant. Hence we get the following equivalence of $A(x)$-modules: 
$$
\rad\big(H^{2b_x}(\B_x,k)^\vee\big)\cong H^{2b_x}_c(\widetilde S_{x,0}\setminus\B_x,k)\cong H^{2b_x+1}_c(\widetilde S_{x,0}\setminus\B_x,k)^\vee\cong H^{2b_x-1}(\widetilde S_{x,0}\setminus\B_x,k).
$$

\subsection{Category of perverse sheaves} For a topological space $X$, we denote by $D^b(X,k)$ the bounded derived category
of $k$-sheaves. For an object $E$ of $D^b(X,k)$, we denote by $H^m(E)$ its $m$th cohomology sheaf.
For a stratification $\mathcal S$ of $X$ with locally closed equidimensional strata, we denote by $D^b_{\mathcal S}(X,k)$
the full subcategory of $D^b(X,k)$ whose objects $E$ are such that each cohomology sheaf $H^m(E)$
is ${\mathcal S}$-constructible (that is, the restriction $i_S^*H^m(E)$ to every stratum $S\in\mathcal S$ is locally constant).

We consider only the middle perversity $p:\mathcal S\to\Z$, which is defined by $p(S)=-\dim S$, and denote
by ${}^p\tau_{\ge 0}$ and ${}^p\tau_{\le 0}$ the truncation functors defined by the $t$-structure relative to $p$.
We denote by $\Hp H^0:D^b(X,k)\to D^b(X,k)$ the cohomological functor ${}^p\tau_{\ge 0}{}^p\tau_{\le 0}$.
In the case of trivial stratification $\mathcal S=\{X\}$, we have the following relation:
$$
\Hp H^0 E=(H^{-\dim X}E)[\dim X].
$$
Considering the shifted functor $\Hp H^a E:=\Hp H^0(E[a])$, we get in this case
\begin{equation}\label{eq:0}
\Hp H^a E=(H^{-\dim X}E[a])[\dim X]=(H^{-\dim X+a}E)[\dim X].
\end{equation}

The core of the $t$-structure for $D^b_\S(X,k)$ relative to perversity $p$ is the category of {\it perverse sheaves} $P_\S(X,k)$.
Following~\cite{BBD}, we set $\Hp\,T=\Hp H^0\circ T\circ\epsilon'$ for any exact functor between triangulated categories $T:D^b_{\S'}(X',k)\to D^b_\S(X,k)$,
where $\epsilon':P_{\S'}(X',k)\to D^b_{\S'}(X',k)$ is the natural inclusion.
We get in this way, an additive functor from $P_{\S'}(X',k)$ to $P_\S(X,k)$.

In what follows, we consider only the stratification $\mathfrak X=\{\O_x\}_{x\in\Lm}$ of $\Nc$ and
the stratifications $\mathfrak X|_Z=\{\O_x\}_{x\in\Lm'}$ of unions $Z=\bigsqcup_{x\in\Lm'}\O_x$, where $\Lambda'\subset\Lambda$.
In this case, we use the notation $P_{\mathfrak X}(Z,k):=P_{\mathfrak X|_Z}(Z,k)$ for brevity.

The simple objects of the category $P_{\mathfrak X}(Z,k)$ are the intersection cohomology complexes
$\IC_Z(\overline\O_x\cap Z,\L):=i_{!*}\L[d_x]$, where $\L$ is an irreducible local system on $\O_x\subset Z$ and
$i:\O_x\to Z$ is the natural inclusion. In Section~\ref{Composition multiplicities_and_semisimplicity}, we study the composition
multiplicities of $\IC_Z(\overline\O_x\cap Z,\L)$ in objects $E$
of $P_{\mathfrak X}(Z,k)$, which we denote by $[E:\IC_Z(\overline\O_x\cap Z,\L)]$.

In the proof of Theorem~\ref{theorem:1}, we shall also need the dualizing functor $\DD:=R\Hom(\_,\D)$,
where $\D$ is dualizing complex for the corresponding space.

\section{Neighbourhood bases}\label{Neighbourhood_bases}
\subsection{Equivariant transverse slices}
We consider here a $C$-equivariant version of Lem\-ma~3.2.20 from~\cite{CG97}, where $C$ is a reductive subgroup
of the centralizer. Remember the basic constructions of~\cite{CG97}. Let $G$ be an algebraic group over $\C$, $V$ be a
smooth algebraic $G$-variety, $X$ be a $G$-stable algebraic subvariety in $V$ and $\mathbb O$ be a $G$-orbit in $X$.
Take a submanifold $S_V$ containing a point $y\in X$ such that $T_yV=T_y\mathbb O\oplus T_yS_V$.
Suppose that $C$ is a reductive subgroup of $Z_G(y)$ such that $S_V$ is $C$-stable.

\begin{lemma}\label{lemma:0}
There exist open neighbourhoods $\Sigma_V\subset S_V$ and 
$U\subset X$ of $y$ and a homeomorphism $f:(\mathbb O\cap U)\times(X\cap \Sigma_V)\ito U$ such that
{
\renewcommand{\labelenumi}{{\rm \theenumi}}
\renewcommand{\theenumi}{{\rm(\roman{enumi})}}
\begin{enumerate} 
\item\label{part:1:lemma:0} $f$ restricts to projections
      $$
         {\rm pr}_2:\{y\}\times(X\cap \Sigma_V)\to X\cap\Sigma_V,\qquad {\rm pr}_1:(\mathbb O\cap U)\times\{y\}\to \mathbb O\cap U
      $$
\item\label{part:2:lemma:0} for any $c\in C$, there exist open neighbourhoods $Q_c\subset\mathbb O\cap U$ and
      $S_c\subset X\cap\Sigma_V$ of~$y$
      such that $c\cdot Q_c\subset\mathbb O\cap U$, $c\cdot S_c\subset X\cap\Sigma_V$ and
      $f(c\cdot q,c\cdot s)=c\cdot f(q,s)$ for any $q\in Q_c$ and $s\in S_c$.
\end{enumerate}}
\end{lemma}
\begin{proof} As $C$ is reductive, $\mathfrak g$ is a rational $C$-module and we work over $\C$, the subspace
$\mathfrak p\subset\mathfrak g$ in the decomposition $\mathfrak p\oplus\Lie Z_G(y)=\mathfrak g$ can be chosen $C$-stable.
Thus $c\,\exp(\mathfrak p)\,c^{-1}=\exp(\mathfrak p)$ for any $c\in C$. Then the proof of part~\ref{part:1:lemma:0} goes
unchanged as in~\cite[Lemma~3.2.20]{CG97}. Moreover, it was noted in this proof the action map $g\mapsto g\cdot y$ takes
some open neighbourhood $P\subset\exp(\mathfrak p)$ of $\mathbf 1_G$ homeomorphically to $\mathbb O\cap U$.

Now choose an element $c\in C$. Consider the following sets
$$
P_c:=P\cap c^{-1}Pc,\quad Q_c:=P_c\cdot y,\quad \Sigma_{V,c}:=\Sigma_V\cap c^{-1}\cdot \Sigma_V,\quad S_c:=\Sigma_{V,c}\cap X.
$$
The inclusions $c\cdot Q_c\subset\mathbb O\cap U$ and $c\cdot S_c\subset X\cap \Sigma_V$ are obvious as $c\in C\subset Z_G(y)$.

Let $q\in Q_c$ and $s\in S_c$. We have $q=p\cdot y$ for some $p\in P_c$. Thus $c\cdot q=cpc^{-1}\cdot y$.
As $cpc^{-1}\in P$ and $f$ is induced by the action map (see the proof of~\cite[Lemma~3.2.20]{CG97}), we get
$$
f(c\cdot q,c\cdot s)=cpc^{-1}c\cdot s=cp\cdot s=c\cdot f(q,s).
$$
\end{proof}

Now we can generalize Corollary 3.2.21 from~\cite{CG97}.
\begin{corollary}\label{corollary:0}
Suppose that additionally to the hypothesis of Lemma~\ref{lemma:0}, we have a continuous $G$-equivariant map
$\pi:\widetilde X\to X$. Set $\widetilde U:=\pi^{-1}(U)$, $\widetilde S:=\pi^{-1}(X\cap \Sigma_V)$ and
$\widetilde S_c:=\pi^{-1}(S_c)$ for any $c\in C$.
There exists a homeomorphism $\tilde f:(\mathbb O\cap U)\times\widetilde S\ito\widetilde U$ such that
{
\renewcommand{\labelenumi}{{\rm \theenumi}}
\renewcommand{\theenumi}{{\rm(\roman{enumi})}}
\begin{enumerate}
\item\label{part:i:corollary:0} $\pi\big(\tilde f(q,\tilde s)\big)=f(q,\pi(\tilde s))$ for any $q\in\mathbb O\cap U$ and $\tilde s\in\widetilde S$.
\item\label{part:ii:corollary:0} for any $c\in C$, we get $c\cdot \widetilde S_c\subset\widetilde S$ and
$\tilde f(c\cdot q,c\cdot\tilde s)=c\cdot f(q,\tilde s)$ for any $q\in Q_c$ and $\tilde s\in\widetilde S_c$.
\end{enumerate}}
\end{corollary}

\begin{remark}\label{remark:1}\rm
Given open neighbourhoods $\Sigma'_V\subset\Sigma_V$ and $Q'\subset\mathbb O\cap U$ of $y$,
we can shrink $U$ to $U':=f(Q'\times(X\cap \Sigma'_V))$ and consider the restriction $f':=f\big|_{Q'\times(X\cap \Sigma'_V)}$.
As $f$ is a homeomorphism, $U'$ is open in $X$. Moreover, $\mathbb O\cap U'=Q'$.
\end{remark}

If $\pi$ is as in Corollary~\ref{corollary:0}, then we set $\widetilde S':=\pi^{-1}(X\cap\Sigma'_V)$,
$\widetilde U':=\tilde f(Q'\times\widetilde S')$ and $\tilde f':=\tilde f\big|_{Q'\times\widetilde S'}$.
Note that $\widetilde U':=\pi^{-1}(U')$. So we have the following commutative diagrams:
$$
\begin{tikzcd}
(\mathbb O\cap U)\times(X\cap \Sigma_V)\arrow{r}[swap]{\sim}{f}&U\\
\arrow[hook]{u}&\arrow[hook]{u}\\[-32pt]
(\mathbb O\cap U')\times(X\cap \Sigma'_V)\arrow{r}{f'}[swap]{\sim}&U'
\end{tikzcd}
\qquad\quad
\begin{tikzcd}
(\mathbb O\cap U)\times\widetilde S\arrow{r}[swap]{\sim}{\tilde f}&\widetilde U\\
\arrow[hook]{u}&\arrow[hook]{u}\\[-32pt]
(\mathbb O\cap U')\times\widetilde S'\arrow{r}{\tilde f'}[swap]{\sim}&\widetilde U'
\end{tikzcd}
$$

For any element $c\in C$, we set
$$
Q'_c:=Q'\cap c^{-1}\cdot Q'\cap Q_c,\quad S'_c:=(X\cap \Sigma'_V)\cap c^{-1}\cdot (X\cap \Sigma'_V)\cap S_c,\quad \widetilde S'_c:=\pi^{-1}(S'_c).
$$
One can easily check that these new data satisfy all conditions of Lemma~\ref{lemma:0} and Corollary~\ref{corollary:0}.

\begin{remark}\label{remark:2}\rm Applying properties~(i) of Lemma~\ref{lemma:0} and Corollary~\ref{corollary:0}, one can easily check that
$$
\tilde f\big((\mathbb O\cap U)\times\pi^{-1}(X\cap \Sigma_V\setminus\{y\})\big)=\pi^{-1}(U\setminus\mathbb O).
$$
\end{remark}

\subsection{Neighbourhoods in slices}
We describe here a special neighbourhood basis of $x$ in $S_{x,0}$.
First, recall the following formula from~\cite[\sectsign 2.5]{Slodowy_Four_Lectures}:
$$
t*\Big(x+\sum_{i=1}^{r'}c_ie_i\Big)=x+\sum_{i=1}^{r'}t^{n_i+1}c_ie_i,
$$
where $e_1,\ldots,e_{r'}$ is a basis of $\ker\ad\,y$, $n_i$ are nonnegative integers,
$c_i\in\C$ and $t\in(0,+\infty)$. It follows from this formula that the action of $C(x,y)$ on both $S_{x,0}$
and $\widetilde S_{x,0}$ is $*$-invariant, that is
\begin{equation}\label{eq:-4}
c\cdot t*z=t*c\cdot z
\end{equation}
for any $c\in C(x,y)$, $t\in(0,+\infty)$ and $z\in S_{x,0}$ or $z\in\widetilde S_{x,0}$.

Many statements below can be proved using the following observation.
Suppose that $\{t_n\}$ is a sequence of numbers in $(0,+\infty)$ and $s\in S_x$. If $\lim_{n\to+\infty}t_n=0$,
then $\lim_{n\to+\infty}t_n*s=x$, and if $\lim_{n\to+\infty}t_n=t>0$, then $\lim_{n\to+\infty}(t_n*s)=t*s$.

Consider the following open neighbourhood (``open ball'') of $x$ in $S_x$:
$$
B:=\bigg\{x+\sum_{i=1}^{r'}c_ie_i\,\Big|\,\sum_{i=1}^{r'}|c_i|^2<1\bigg\}.
$$
Clearly, $\overline{t*B}\subset B$ for $t<1$. Moreover, the closure $\overline B$ is compact and
$\{t*B\suchthat t\in(0,+\infty)\}$ is a neighbourhood basis for $x$ in $S_x$.
The intersection $A:=B\cap\Nc$ satisfies the following properties:
\begin{enumerate}
\itemsep=4pt
\item\label{A:prop:1} The closure $\overline A=\overline B\cap\Nc$ is compact.
\item\label{A:prop:2} $\{t*A\suchthat t\in(0,+\infty)\}$ is a neighbourhood basis for $x$ in $S_{x,0}$.
\item\label{A:prop:3} $\overline{t_0*A}\subset t_1*A$ for $t_0<t_1$.
\end{enumerate}

\begin{lemma}\label{lemma:1}
For $t_0<t_1$, the natural inclusion and the $*$-action by $t_1/t_0$
are homotopic maps from $\pi^{-1}(t_0*A\setminus\{x\})$ to $\pi^{-1}(t_1*A\setminus\{x\})$. Each such space
is homeomorphic to $\widetilde S_{x,0}\setminus\B_x$.
\end{lemma}
\begin{proof}
The required homotopy $F:\pi^{-1}(t_0*A\setminus\{x\})\times[t_0,t_1]\to\pi^{-1}(t_1*A\setminus\{x\})$ is given by
$F(\tilde s,t):=(t/t_0)*\tilde s$.

To prove the second statement, we define the {\it radius} of a point $s\in S_{x,0}$ by
$$
r(s):=\inf\big\{t\in(0,+\infty)\suchthat s\in t*A\big\}.
$$
From this definition, it is clear that $r(t*s)=tr(s)$ for any $t\in(0,+\infty)$.
Property~\ref{A:prop:3} above shows that $r$ is continuous and that $r(s)<1$ if and only if $s\in A$.
By property~\ref{A:prop:2}, we get that $r(s)=0$ if and only if $s=x$.

Now let $\sigma:(0,1)\to(0,+\infty)$ be the function defined by $\sigma(a):=a/(1-a)$.
It induces the function $\zeta:\pi^{-1}(A\setminus\{x\})\to(0,+\infty)$ by $\zeta(\tilde s):=\sigma\circ r\circ \pi(\tilde s)$.
We define the homeomorphism
$u:\pi^{-1}(A\setminus\{x\})\to\widetilde S_{x,0}\setminus\B_x$ by
$
u(\tilde s)=\zeta(\tilde s)*\tilde s
$.
To calculate its inverse, consider the function $\tau:(0,+\infty)\to(0,1)$ given by $\tau(a):=(-a+\sqrt{a^2+4a})/2$.
Then the inverse map is given by $u^{-1}(\tilde s)=\big(\sigma\circ \tau\circ r\circ\pi(\tilde s)\big)^{-1}*\tilde s$.
\end{proof}

\subsection{Equivariant transverse slices for nilpotent cones}\label{eq_nilpotent_cones} To construct a neighbourhood basis for $x$ in $\Nc$,
we apply Lemma~\ref{lemma:0} in the following situation:
$y=x$, $V=\g$, $\mathbb O=\O_x$, $X=\Nc$, $S_V=S_x$ and $C=C(x,y)=Z_G(x)\cap Z_G(y)$. Taking in to account Remark~\ref{remark:1},
we conclude that for small enough $t$, there exists an open neighbourhood $U_x(t)\subset\Nc$ of $x$ and a homeomorphism
$$
f_x(t):(\O_x\cap U_x(t))\times(t*A)\ito U_x(t)
$$
satisfying properties~\ref{part:1:lemma:0} and~\ref{part:2:lemma:0} of Lemma~\ref{lemma:0}. Applying Remark~\ref{remark:1}
one more time, we can suppose that $\O_x\cap U_x(t)$ is homeomorphic
to the open ball $\Ball(t)$ of radius $t$ centered at $0$ in a way respecting inclusion:
for $t_1<t_2$ the following diagram is commutative:
$$
\begin{tikzcd}
\O_x\cap U_x(t_2)\arrow{r}{\sim}&\Ball(t_2)\\
\arrow[hook]{u}&\arrow[hook]{u}\\[-32pt]
\O_x\cap U_x(t_1)\arrow{r}{\sim}&\Ball(t_1)
\end{tikzcd}
$$
The open sets $U_x(t)$ clearly form a neighbourhood basis of $x$ in $\Nc$.

Applying Corollary~\ref{corollary:0} to the Springer resolution $\pi:\widetilde\Nc\to\Nc$, we get the homeomorphism
$$
\tilde f_x(t):(\O_x\cap U_x(t))\times\pi^{-1}(t*A)\ito\pi^{-1}(U_x(t))
$$
satisfying properties~\ref{part:i:corollary:0} and~\ref{part:ii:corollary:0} of this corollary.
Restricting to $(\O_x\cap U_x(t))\times\pi^{-1}(t*A\setminus\{x\})$ and applying Remark~\ref{remark:2},
we obtain the homeomorphism
$$
\hat f_x(t):(\O_x\cap U_x(t))\times\pi^{-1}(t*A\setminus\{x\})\ito\pi^{-1}(U_x(t)\setminus\O_x).
$$
For $t_1<t_2$, the commutativity of the diagram
$$
\begin{tikzcd}
(\O_x\cap U_x(t_2))\times\pi^{-1}({t_2}*A\setminus\{x\})\arrow{r}[swap]{\sim}{\hat f_x(t_2)}&\pi^{-1}(U_x(t_2)\setminus\O_x)\\
\arrow[hook]{u}&\arrow[hook]{u}\\[-32pt]
(\O_x\cap U_x(t_1))\times\pi^{-1}({t_1}*A\setminus\{x\})\arrow{r}[swap]{\sim}{\hat f_x(t_1)}&\pi^{-1}(U_x(t_1)\setminus\O_x)
\end{tikzcd}
$$
together with Lemma~\ref{lemma:1} prove that the natural inclusion
$\pi^{-1}(U_x(t_1)\setminus\O_x)\hookrightarrow\pi^{-1}(U_x(t_2)\setminus\O_x)$ is homotopic to a homeomorphism for $t_0<t_1$.

\section{Composition multiplicities and semisimplicity}\label{Composition multiplicities_and_semisimplicity}

For any $kA(x)$-module $N$ and an irreducible $kA(x)$-module $M$, we denote by $[N:M]$ the composition multiplicity of $M$ in $N$.
The main result of this paper is as follows.

\begin{theorem}\label{theorem:1} For any $x\in\Lm$ and an irreducible $kA(x)$-module $M$, we have
$$
\Big[\pi_*\csh k{\NN}[d_\Nc]:\IC_\Nc\big(\overline\O_x,\Loc(M)\big)\Big]\ge \big[H^{2b_x}(\B_x,k):M\big]+\big[\rad\big(H^{2b_x}(\B_x,k)^\vee\big):M\big].
$$

Moreover, $\pi_*\csh k{\NN}[d_\Nc]$ is semisimple in $P_{\mathfrak X}(\Nc,k)$ if and only if
$\rad\big(H^{2b_x}(\B_x,k)^\vee\big)=0$ and $H^{2b_x}(\B_x,k)$ is semisimple as a $kA(x)$-module
for any $x\in\Lm$.
\end{theorem}
\begin{proof}
We denote for brevity $E:=\pi_*\csh k{\NN}[d_\Nc]$. Take any nonempty open subset $\Phi\subset\Lambda$.  As we have the
cartesian diagram
$$
\begin{CD}
\NN_\Phi@>\tilde\imath_\Phi>>\NN\\
@V\pi_\Phi VV @V\pi VV\\
\Nc_\Phi@>i_\Phi>>\Nc
\end{CD}
$$
and the maps $\pi$ and $\pi_\Phi$ are proper, the base change gives
\begin{equation}\label{eq:1}
i_\Phi^*\,E=i_\Phi^*\,\pi_*\,\csh k{\NN}[d_\Nc]=\pi_{\Phi*}\,\tilde\imath_\Phi^{\,*}\,\csh k{\NN}[d_\Nc]=\pi_{\Phi*}\,\csh k{\NN_\Phi}[d_\Nc].
\end{equation}
We denote this complex by $E_\Phi$ and prove inductively on $\Phi$ with respect to the inclusion relation the following statement:

{\bf Inductive claim.} {\it
For any $x\in\Phi$ and an irreducible $kA(x)$-module $M$, we have
\begin{equation}\label{eq:11}
\Big[E_\Phi:\IC_{\Nc_\Phi}\big(\overline\O_x\cap\Nc_\Phi,\Loc(M)\big)\Big]\ge\big[H^{2b_x}(\B_x,k):M\big]+\big[\rad\big(H^{2b_x}(\B_x,k)^\vee\big):M\big].
\end{equation}

Moreover, $E_\Phi$ is semisimple in $P_{\mathfrak X}(\Nc_\Phi,k)$ if and only if
$\rad\big(H^{2b_x}(\B_x,k)^\vee\big)=0$ and \linebreak$H^{2b_x}(\B_x,k)$ is semisimple as a $kA(x)$-module
for any $x\in\Phi$.}

The case where $\Phi$ consists of one regular element $x_0$ is obvious, as $\pi_{\{x_0\}}$
is a homeomorphism. Hence $E_\Phi=\csh k{\O_{x_0}}[d_\Nc]=\IC_{\O_{x_0}}(\O_{x_0},k)$
is a simple module. Clearly $H^{2b_{x_0}}(\B_{x_0},k)=H^0(\{\pt\},k)=k$ and
$\rad\big(H^{2b_{x_0}}(\B_{x_0},k)^\vee\big)=H^{-1}(\widetilde S_{x_0,0}\setminus\B_{x_0},k)=0$.
It remains to notice that $k$ is the only irreducible $kA(x_0)$-module as $A(x_0)$ is the trivial group.

Now suppose that $\Phi$ contains at least two elements. Choose a minimal point $x\in\Phi$ and consider the open set
$\Psi:=\Phi\setminus\{x\}$. Let $i:\O_x\to\Nc_\Phi$ and $j:\Nc_\Psi\to\Nc_\Phi$ be the natural closed and open
embeddings, respectively.

{\it Step 1.} We have the following exact sequence (see~\cite[Lemme 1.4.19]{BBD}):
\begin{equation}\label{eq:2}
0\to{}^pi_*\,\Hp H^{-1}\,i^*\,E_\Phi\to{}^pj_!\;{}^pj^*\,E_\Phi\to E_\Phi\to{}^pi_*\;{}^pi^*\,E_\Phi\to0
\end{equation}
First, let us look at its last term.
Consider the following cartesian diagram:
$$
\begin{CD}
\NN_{\{x\}}@>\tilde\imath_{\{x\}}>>\NN\\
@V\pi_{\{x\}}VV @VV\pi V\\
\O_x@>i_{\{x\}}>>\Nc
\end{CD}
$$
Now again, the base change gives
$$
i^*\,E_\Phi=(i_\Phi i)^*\,E=i_{\{x\}}^*\,\pi_*\,\csh k{\NN}[d_\Nc]=\pi_{\{x\}*}\,\tilde\imath_{\{x\}}^{\,*}\,\csh k{\NN}[d_\Nc]=\pi_{\{x\}*}\,\csh k{\NN_{\{x\}}}[d_\Nc].
$$
This formula and~(\ref{eq:0}) imply
\begin{multline*}
{}^pi^*\,E_\Phi=\Hp H^0(\pi_{\{x\}*}\,\csh k{\NN_{\{x\}}}[d_\Nc])=\Big(H^{-d_x}(\pi_{\{x\}*}\,\csh k{\NN_{\{x\}}}[d_\Nc])\Big)[d_x]\\
=\bigg(H^{d_\Nc-d_x}\pi_{\{x\}*}\,\csh k{\NN_{\{x\}}}\bigg)[d_x]=\bigg(R^{\,2b_x}\pi_{\{x\}*}\,\csh k{\NN_{\{x\}}}\bigg)[d_x].
\end{multline*}

We denote $\L:=R^{\,2b_x}\pi_{\{x\}*}\,\csh k{\NN_{\{x\}}}$. This is a local system on $\O_x$.
As $\pi$ is proper, we get
$\L_x=H^{2b_x}(\B_x,k)$. A well-known argument shows that
\begin{equation}\label{eq:3}
\L\cong\Loc(H^{2b_x}(\B_x,k)).
\end{equation}
We are going to give a brief outline of this proof. Let $\delta:[0,1]\to G$ be a curve such that $\delta(0)=\mathbf 1_G$.
We set $c:=\delta(1)$. We define the map $\beta:G\to\O_x$ by $\beta(g)=g\cdot x$.
The composition $\gamma:=\beta\circ\delta$ is a curve in $\O_\lm$.
We get $x:=\gamma(0)$ and additionally assume that $x=\gamma(1)$.
Hence $c\in Z_G(x)$.

A routine check involving the fact that homotopic maps induce the same maps on cohomology
spaces\footnote{See Section~\ref{structureiM} for a similar argument.} 
yields the commutativity of the following diagram:
$$
\begin{tikzcd}
H^{2b_x}(\B_x,k)\arrow{rrrr}[above]{m_{c^{-1}}^*}&&&&H^{2b_x}(\B_x,k)\\[-35pt]
\arrow[swap]{d}{\wr}&&&&\arrow[swap]{d}{\wr}\\
\L_x\arrow{r}{\sim}&(\gamma^*\L)_0\arrow{r}{}&\gamma^*\L([0,1])\arrow{r}{}&(\gamma^*\L)_1\arrow{r}{\sim}&\L_x
\end{tikzcd}
$$
The bottom line corresponds to the homotopy action of the loop $\gamma$. As noted in Section~\ref{local_systems},
$\alpha$ maps the image of $\gamma$ in $\Pi_1(\O_x)$ to $cZ_G(x)^0$. Hence in the notation of~Section~\ref{local_systems},
we get $\L_x\cong H^{2b_x}(\B_x,k)^\alpha$, whence~(\ref{eq:3}) follows.

Now let us look at the first term of~(\ref{eq:2}). Similarly to the above calculations for ${}^pi^*\,E_\Phi$, we get
$$
\Hp H^{-1}\,i^*\,E_\Phi=\Hp H^{-1}\pi_{\{x\}*}\,\csh k{\NN_{\{x\}}}[d_\Nc]=\bigg(R^{\,2b_x-1}\pi_{\{x\}*}\,\csh k{\NN_{\{x\}}}\bigg)[d_x].
$$
As $\pi_{\{x\}}$ is proper, we get
\begin{equation}\label{eq:-2}
\bigg(R^{\,2b_x-1}\pi_{\{x\}*}\,\csh k{\NN_{\{x\}}}\bigg)_{\!y}=H^{\,2b_x-1}(\B_y,k)=0.
\end{equation}
for any $y\in\O_x$.
Thus we have proved that $\Hp H^{-1}\,i^*\,E_\Phi=0$.

Note here that $j^*E_\Phi={}^pj^*E_\Phi$, as $j^*$ is t-exact, and
${}^pj^*E_\Phi=(i_\Phi j)^*\,E=i_\Psi^*\,E=E_\Psi$.
So, using~(\ref{eq:3}) and~(\ref{eq:-2}), we can rewrite~(\ref{eq:2}) as the following exact sequence:
\begin{equation}\label{eq:9}
\begin{tikzcd}[column sep=4ex]
0\arrow{r}{}&{}^pj_!\,E_\Psi\arrow{r}{}&E_\Phi\arrow{r}{}&\IC_{\Nc_\Phi}\big(\O_x,\Loc(H^{2b_x}(\B_x,k))\big)\arrow{r}{}&0.
\end{tikzcd}
\end{equation}
Note that
\begin{equation}\label{eq:II}
\begin{array}{l}
\!\!\!\!\!\!\Big[\IC_{\Nc_\Phi}\big(\O_x,\Loc(H^{2b_x}(\B_x,k))\big):\IC_{\Nc_\Phi}(\O_x,\Loc(M))\Big]\\[12pt]
\hspace{15pt}=\Big[\Loc(H^{2b_x}(\B_x,k)):\Loc(M)\Big]=\big[H^{2b_x}(\B_x,k)^\alpha:M^\alpha\big]=\big[H^{2b_x}(\B_x,k):M\big].
\end{array}
\end{equation}
for any irreducible $kA(x)$-module $M$.

{\it Step 2.} We have the following exact sequences:
\begin{equation}\label{eq:4}
\begin{tikzcd}[column sep=4.5ex, row sep=small]
0\arrow{r}&i_*\Hp H^{-1}i^*j_*E_\Psi\arrow{r}&{}^pj_!E_\Psi\arrow{rd}\arrow{rr}&&{}^pj_*E_\Psi\arrow{r}&i_*{}^pi^*j_*E_\Psi\arrow{r}&0\\[-12pt]
          &                             &                           &j_{!*}E_\Psi\arrow{ru}\arrow{rd}\\[-12pt]
          &                             &0\arrow{ru}                           &                 &0
\end{tikzcd}\!\!
\end{equation}

Our next aim is to calculate $\Hp H^{-1}\,i^*j_*E_\Psi$. 
Let $\xi$ denote the composition 
$$
\begin{CD}
\NN_\Psi@>\pi_\Psi>>\Nc_\Psi@>j>>\Nc_\Phi.
\end{CD}
$$
We have $E_\Psi=\pi_{\Psi*}\,\csh k{\NN_\Psi}[d_\Nc]$ as defined by~(\ref{eq:1}), where $\Phi$ is replaced by $\Psi$.
Thus applying~(\ref{eq:0}), we get
\begin{equation}\label{eq:c}
\begin{array}{l}
\Hp H^{-1}\,i^*j_*\,E_\Psi
=\Hp H^{-1}\big(i^*(j\,\pi_\Psi)_*\,\csh k{\NN_\Psi}[d_\Nc]\big)=\bigg(H^{-1-d_x}(\,i^*\xi_*\,\csh k{\NN_\Psi}[d_\Nc])\bigg)[d_x]\\
\hspace{150pt}=\bigg(i^*H^{d_\Nc-1-d_x}\xi_*\,\csh k{\NN_\Psi}\bigg)[d_x]
=\bigg(i^*R^{2b_x-1}\xi_*\,\csh k{\NN_\Psi}\bigg)[d_x].
\end{array}
\end{equation}

By proposition~\cite[Proposition II.5.11]{Iversen}, the sheaf $\M:=R^{2b_x-1}\xi_*\,\csh k{\NN_\Psi}$
in the right-hand side is the sheaf associated to the presheaf $\M^\pre$ defined by
$$
U\mapsto H^{\,2b_x-1}(\xi^{-1}(U),k)
$$
for open $U\subset\Nc_\Phi$. By~(\ref{eq:c}) and equivalence~(\ref{eq:I}) proved (independently) in Section~\ref{structureiM} ,
we get
$
\Hp H^{-1}\,i^*j_*\,E_\Psi=\Loc\big(\rad\big(H^{2b_x}(\B_x,k)^\vee\big)\big)[d_x].
$
Finally, applying $i_*$, we obtain
\begin{equation}\label{eq:IV}
i_*\Hp H^{-1}\,i^*j_*E_\Psi=\IC_{\Nc_\Phi}\Big(\O_x,\Loc\big(\rad\big(H^{2b_x}(\B_x,k)^\vee\big)\big)\Big).
\end{equation}
Using this formula and coming back to diagram~(\ref{eq:4}), we get for any irreducible $kA(x)$-module $M$ the following inequalities:
\begin{equation}\label{eq:III}
\begin{array}{l}
\Big[{}^pj_!E_\Psi:\IC_{\Nc_\Phi}(\O_x,\Loc(M))\Big]\ge\Big[i_*\Hp H^{-1}\,i^*j_*E_\Psi:\IC_{\Nc_\Phi}(\O_x,\Loc(M))\Big]\\[12pt]
\hspace{30pt}=\Big[\Loc\big(\rad\big(H^{2b_x}(\B_x,k)^\vee\big)\big):\Loc(M)\Big]=\Big[\rad\big(H^{2b_x}(\B_x,k)^\vee\big)^\alpha:M^\alpha\Big]\\[12pt]
\hspace{270pt}=\Big[\rad\big(H^{2b_x}(\B_x,k)^\vee\big):M\Big],
\end{array}
\end{equation}
where we used the notation of Section~\ref{local_systems}.

Using sequences~(\ref{eq:9}) and~(\ref{eq:4}) and applying the functor $j_{!*}$ (which is known to preserve monomorphisms
and epimorphisms), by the inductive hypothesis, we get
\begin{multline*}
\Big[E_\Phi:\IC_{\Nc_\Phi}\big(\overline\O_y\cap\Nc_\Phi,\Loc(N)\big)\Big]
\ge\Big[{}^pj_!\,E_\Psi:\IC_{\Nc_\Phi}\big(\overline\O_y\cap\Nc_\Phi,\Loc(N)\big)\Big]\\
\ge
\Big[j_{!*}E_\Psi:\IC_{\Nc_\Phi}\big(\overline\O_y\cap\Nc_\Phi,\Loc(N)\big)\Big]
\ge\Big[E_\Psi:\IC_{\Nc_\Psi}\big(\overline\O_y\cap\Nc_\Psi,\Loc(N)\big)\Big]\\
\ge\big[H^{2b_y}(\B_y,k):N\big]+\big[\rad\big(H^{2b_y}(\B_y,k)^\vee\big):N\big]
\end{multline*}
for any $y\in\Psi$ and irreducible $kA(y)$-module $N$. Using these inequalities and plugging inequalities~(\ref{eq:II})
and~(\ref{eq:III}) to sequence~(\ref{eq:9}), we get the inequalities of the inductive claim for $\Phi$.

{\it Step 3}. Suppose that $E_\Phi$ is semisimple.
By sequence~(\ref{eq:9}), we get that ${}^pj_!\;E_\Psi$ is also semisimple.
We know that ${}^pj_!\;E_\Psi$ does not have quotient objects in $i_*P_\X(\O_x,k)$.
By semisimplicity, ${}^pj_!\;E_\Psi$ also does not have subobjects in $i_*P_\X(\O_x,k)$.
Thus ${}^pj_!\;E_\Psi=j_{!*}E_\Psi$ and by diagram~(\ref{eq:4}), we get $i_*\Hp H^{-1}i^*j_*E_\Psi=0$.
Hence $\rad\big(H^{2b_x}(\B_x,k)^\vee\big)=0$ by~(\ref{eq:IV}).

Coming back to sequence~(\ref{eq:9}),
we get that the quotient $\IC_{\Nc_\Phi}\big(\O_x,\Loc(H^{2b_x}(\B_x,k))\big)$ is also semisimple.
Thus $\Loc(H^{2b_x}(\B_x,k))$ is semisimple. Hence the $\Pi_1(\O_x)$-module $H^{2b_x}(\B_x,k)^\alpha$ is semisimple
and the $kA(x)$-module $H^{2b_x}(\B_x,k)$ is also semisimple.
As $E_\Psi$ is also semisimple, we get by induction that $\rad\big(H^{2b_y}(\B_y,k)^\vee\big)=0$ and
$H^{2b_y}(\B_x,y)$ is semisimple as a $kA(x)$-module for any $y\in\Psi$.

{\it Step 4}. Suppose that $\rad\big(H^{2b_y}(\B_y,k)^\vee\big)=0$ and $H^{2b_y}(\B_x,y)$ is semisimple as a $kA(x)$-module
for any $y\in\Phi$. By~(\ref{eq:4}) and~(\ref{eq:IV}),
we have ${}^pj_!E_\Psi=j_{!*}E_\Psi$. So sequence~(\ref{eq:9}) can be rewritten as
\begin{equation}\label{eq:13}
\begin{tikzcd}[column sep=4ex]
0\arrow{r}{}&j_{!*}E_\Psi\arrow{r}{\mu}&E_\Phi\arrow{r}{\eta}&\IC_{\Nc_\Phi}\big(\O_x,\Loc(H^{2b_x}(\B_x,k))\big)\arrow{r}{}&0.
\end{tikzcd}
\end{equation}
By induction, $E_\Psi$ is semisimple. Hence so is $j_{!*}E_\Psi$ as well as $\IC_{\Nc_\Phi}\big(\O_x,\Loc(H^{2b_x}(\B_x,k))\big)$.

It remains to show that the above sequence splits. First note that from~(\ref{eq:1}) it follows that $\DD E_\Phi=E_\Phi$.
Moreover, $H^{2b_x}(\B_x,k)^\vee\cong H^{2b_x}(\B_x,k)$, as $\rad\big(H^{2b_x}(\B_x,k)^\vee\big)=0$. Thus
\begin{multline*}
\DD\Big(\IC_{\Nc_\Phi}\big(\O_x,\Loc(H^{2b_x}(\B_x,k))\big)\Big)\cong\IC_{\Nc_\Phi}\big(\O_x,\Loc(H^{2b_x}(\B_x,k))^\vee\big)\\
\cong\IC_{\Nc_\Phi}\big(\O_x,\Loc(H^{2b_x}(\B_x,k)^\vee)\big)\cong\IC_{\Nc_\Phi}\big(\O_x,\Loc(H^{2b_x}(\B_x,k))\big).
\end{multline*}

So we have the dual morphism $\DD\eta:\IC_{\Nc_\Phi}\big(\O_x,\Loc(H^{2b_x}(\B_x,k))\big)\to E_\Phi$. As $\eta$ is an epimorphism,
$\DD\eta$ is a monomorphism. The kernel of the composition $\eta\circ\DD\eta$ is zero as otherwise $\ker\eta$
would have a subobject isomorphic to $\IC_{\Nc_\Phi}(\O_x,\mathcal K)$ for some simple local system $\mathcal K$ on $\O_x$.
This is however impossible, as $\ker\eta=\im\mu\cong j_{!*}E_\Psi$
and the latter complex can not have such subobjects by definition of $j_{!*}$.
By duality, the cokernel of $\eta\circ\DD\eta$ is also zero, whence $\eta\circ\DD\eta$ is an isomorphism.
\end{proof}

\section{Structure of $i^*\M$}\label{structureiM}
We choose some fixed Slodowy slice $S_x$ for $x$ (this means that we choose a
fixed triple $(x,h,y)$ in the Jacobson-Morozov lemma)
and a neighbourhood basis
$\{U_x(t)\}_{t\in(0,t_0]}$ for $x$, where $t_0$ is a fixed positive number as in Section~\ref{eq_nilpotent_cones}.
Reducing $t_0$, if necessary, we can assume that the closure $\overline{U_x(t_0)}$ is compact and is contained in $\Nc_\Phi$.

Remember that  $\beta:G\to\O_x$ is the map defined by $\beta(g)=g\cdot x$.
Consider a curve $\delta:[0,1]\to G$ such that $\delta(0)=\mathbf 1_G$. We set $c:=\delta(1)$ and assume
additionally that $c\in C(x,y)$. The composition $\gamma:=\beta\circ\delta$ is a curve in $\O_x$, whence also in $\N_\Phi$.
Clearly $x=\gamma(0)$. We assume additionally that $\gamma(1)=x$. So $\gamma$ is a loop in $\O_x$ based at $x$.

We can translate $S_x$ and $U_x(t)$ along $\gamma$ by
\begin{equation}\label{eq:14}
{}^aS_{\gamma(a)}:=\delta(a)\cdot S_x\quad\text{and}\quad {}^aU_{\gamma(a)}(t):=\delta(a)\cdot U_x(t)
\end{equation}
for $a\in[0,1]$. We use here the superscript ${}^a$ to avoid ambiguity: the curve $\gamma$ may (and actually do)
have self-intersections. So neither the slice nor the neighbourhood basis is determined uniquely by the point.
By definition, ${}^0S_x={}^1S_x=S_x$ and ${}^0U_x(t)=U_x(t)$, ${}^1U_x(t)=c\cdot U_x(t)$.

We can consider two pull-back functors $\gamma^*$ and $\gamma^*_\pre$ in the categories of sheaves and presheaves respectively.
Consider the following morphisms of sheafification (see~\cite{Vakil}):
$$
\sh:\M^\pre\to\M\quad\text{ и }\quad\sh':\gamma^*_\pre\M\to\gamma^*\M.
$$
Then the composition $\sh'\circ\gamma^*_\pre\sh:\gamma^*_\pre\M^\pre\to\gamma^*\M$ is also a morphism of sheafification.

Take a section $s\in\gamma^*\M([0,1])$. There exists an open covering $[0,1]=\bigcup_{i\in J}V_i$ such that
$s|_{V_i}=\sh'(V_i)(s_i)$ for some $s_i\in\gamma^*_\pre\M(V_i)$. Now recall that
$$
\gamma^*_\pre\M(V_i)=\varinjlim_{W_i\supset\gamma(V_i)}\M(W_i).
$$
So we can assume that each $s_i$ is represented by a section $r_i\in\M(W_i)$ for some open $W_i\subset\N_\Phi$
containing $\gamma(V_i)$.
Shrinking $W_i$ and refining the covering $[0,1]=\bigcup_{i\in J}V_i$, if necessary, we can assume that $r_i=\sh(W_i)(p_i)$
for some $p_i\in\M^\pre(W_i)=H^{\,2b_x}(\xi^{-1}(W_i),k)$.

As the interval $[0,1]$ is compact, we can additionally assume that $J=\{1,\ldots,m\}$ for a finite $m$, each $V_i$ is convex
and there exist points $a_0<a_1<\cdots<a_m<a_{m+1}\in[0,1]$ such that
$a_0=0\in V_1$, $a_{m+1}=1\in V_m$ и $a_i\in V_i\cap V_{i+1}$ for $i=1,\ldots,m-1$.

Consider the map $\nu_i:\overline{U_x(t_0)}\times[a_i,a_{i+1}]\to\Nc_\Phi$ defined by
$(z,a)\mapsto\delta(a)\cdot z$. It is obviously continuous.
We clearly have
$$
\bigcap_{l>1/t_0,\,l\in\Z}\overline{U_x(1/l)}\times[a_i,a_{i+1}]=\{x\}\times[a_i,a_{i+1}],
$$
whence
$$
W_i\supset\gamma([a_i,a_{i+1}])=\nu_i\Bigg(\bigcap_{l>1/t_0,\,l\in\Z}\overline{U_x(1/l)}\times[a_i,a_{i+1}]\Bigg)=\bigcap_{l>1/t_0,\,l\in\Z}\nu_i\Big(\overline{U_x(1/l)}\times[a_i,a_{i+1}]\Big).
$$
To see why the last equality holds, we use the following topological observation: if $X$ is a countably compact topological space,
$X\supset Z_1\supset Z_2\supset\cdots $ is an infinite sequence of its closed subspace and $f:X\to Y$ is a continuous map,
then $f(\bigcap_{q=1}^{+\infty}Z_q)=\bigcap_{q=1}^{+\infty}f(Z_q)$.
The above inclusion and the compactness of each $\overline{U_x(1/l)}\times[a_i,a_{i+1}]$
prove that $\nu_i(\overline{U_x(1/l)}\times[a_i,a_{i+1}])\subset W_i$ for some $l>1/t_0$.
Restricting this inclusion to $\nu_i(U_x(1/l)\times\{a\})$ and applying~(\ref{eq:14}),
we get ${}^aU_{\gamma(a)}(1/l)\subset W_i$ for any $a\in[a_i,a_{i+1}]$.
Clearly, this $l$ can be chosen the same for all $i=0,\ldots,m$.

For any $a\in[a_i,a_{i+1}]$, the composition of canonical morphisms
\begin{equation}\label{eq:-3}
\gamma^*\M([0,1])\xrightarrow{\hspace{12pt}}(\gamma^*\M)_a\eqto\M_{\gamma(a)}\eqto\M^\pre_{\gamma(a)}\eqto
H^{\,2b_x-1}\big(\pi^{-1}\big({}^aU_{\gamma(a)}(1/l)\setminus\O_x\big),k\big).
\end{equation}
maps $s$ to the image of $p_i$ under the morphism
$$
\iota_a^*:H^{\,2b_x-1}\big(\pi^{-1}(W_i\setminus\O_x),k\big)\to H^{\,2b_x-1}\big(\pi^{-1}\big({}^aU_{\gamma(a)}(1/l)\setminus\O_x\big),k\big),
$$
where $\iota_a:\pi^{-1}\big({}^aU_{\gamma(a)}(1/l)\setminus\O_x\big)\to\pi^{-1}(W_i\setminus\O_x)$
is the natural inclusion. The last isomorphism of~(\ref{eq:-3}) follows from the fact that for any $t\in(0,1/l)$
the natural inclusion $\pi^{-1}({}^aU_{\gamma(a)}(t)\setminus\O_x)\hookrightarrow\pi^{-1}({}^aU_{\gamma(a)}(t)\setminus\O_x)$
is homotopic to a homeomorphism, as noted at the end of Section~\ref{eq_nilpotent_cones}\footnote{Use the shifted
versions $\delta(a)\cdot B$ and $\delta(a)\cdot A$ of $B$ and $A$, respectively.}.

We claim that the following diagram is commutative
$$
\hspace{-9pt}\begin{tikzcd}[column sep=-25pt]
&H^{\,2b_x-1}\big(\pi^{-1}(W_i\setminus\O_x),k\big)\arrow[swap]{dl}{\iota_{a_i}^*}\arrow{dr}{\iota_{a_{i+1}}^*} & \\
H^{\,\,2b_x-1}\big(\pi^{-1}\big({}^{a_i}U_{\gamma(a_i)}(1/l)\setminus\O_x\big),k\big)& & \arrow{ll}{m_{\delta(a_{i+1})\delta(a_i)^{-1}}^*}H^{\,2b_x-1}\big(\pi^{-1}\big(^{a_{i+1}}U_{\gamma(a_{i+1})}(1/l)\setminus\O_x\big),k\big)
\end{tikzcd}\!\!\!\!\!
$$
Indeed, the maps $\iota_{a_i}$ and $\iota_{a_{i+1}}\circ m_{\delta(a_{i+1})\delta(a_i)^{-1}}$
are homotopic, the homotopy being given by $F(t):=\iota_{t}\circ m_{\delta(t)\delta(a_i)^{-1}}$ for $t\in[a_i,a_{i+1}]$.

Hence we get the commutative diagram
\begin{equation}\label{eq:-1}
\begin{tikzcd}[column sep=2.2ex]
H^{\,2b_x-1}\big(\pi^{-1}&\hspace{-17pt}(U_x(1/l)\setminus\O_x),k\big)\arrow{rr}[above]{m_{c^{-1}}^*}&&H^{\,2b_x-1}\big(\pi^{-1}(c\cdot U_x(1/l)\setminus\hspace{-18pt}&\O_x),k\big)\\[-35pt]
\arrow[swap]{d}{\wr}&&&&\arrow{d}{\wr}\\
\M_x\arrow{r}{\sim}&(\gamma^*\M)_0\arrow{r}{}&\gamma^*\M([0,1])\arrow{r}{}&(\gamma^*\M)_1\arrow{r}{\sim}&\M_x
\end{tikzcd}
\end{equation}

To understand what action of $c$ on $\M_x$ we obtained, we need first to
describe explicitly the canonical identification
$$
H^{\,2b_x-1}\big(\pi^{-1}(U_x(1/l)\setminus\O_x),k\big)\ito\M_x\ito H^{\,2b_x-1}\big(\pi^{-1}(c\cdot U_x(1/l)\setminus\O_x),k\big),
$$
which we denote by $\rho$. Choose some $t\in(0,1/l)$ such that $c\cdot U_x(t)\subset U_x(1/l)$. Then $\rho(h)$ is equal to
the preimage of $h|_{\pi^{-1}(c\cdot U_x(t)\setminus\O_x)}$ with respect to the isomorphism
$$
H^{\,2b_x-1}\big(\pi^{-1}(c\cdot U_x(1/l)\setminus\O_x),k\big)\ito H^{\,2b_x-1}\big(\pi^{-1}(c\cdot U_x(t)\setminus\O_x),k\big)
$$
induced by the natural inclusion, which is homotopic to a homeomorphism as proved in Section~\ref{eq_nilpotent_cones}\footnote{Use the shifted versions $c\cdot B$ and
$c\cdot A$ of $B$ and $A$, respectively.}. This construction clearly does not depend on the choice of $t$.

Now consider the homeomorphism $u:\pi^{-1}(A\setminus\{x\})\to\widetilde S_{x,0}\setminus\B_x$ described in Lemma~\ref{lemma:1}.
We can scale it by $u_l(z):=u(l*z)$ to we obtain a homeomorphism from $\pi^{-1}(1/l*A\setminus\{x\})$ to
$\widetilde S_{x,0}\setminus\B_x$. We shall use the explicit formula $u_l(\tilde s)=l\zeta(l*\tilde s)*\tilde s$
following from the proof of Lemma~\ref{lemma:1}.
We also denote by $v$ any homeomorphism from $\O_x\cap U_x(1/l)$ to $\C^{d_x}$.
We claim that the following diagram is commutative in the homotopy category:
\begin{equation}\label{eq:00}
{\!\!\begin{tikzcd}[column sep=29pt]
\pi^{-1}\big(U_x\big(\frac1l\big)\setminus\O_x\big)\arrow{r}{\hat f_x\(\frac 1l\)^{-1}\!\!}&\big(\O_x\cap U_x(\frac 1l)\big)\times\pi^{-1}\big(\frac 1l*A\setminus\{x\}\big)\arrow{r}{v\times u_l}&\C^{d_x}\times\widetilde S_{x,0}\setminus\B_x\\
\pi^{-1}(U_x(t)\setminus\O_x)\arrow[hook]{u}\\
\pi^{-1}(c\cdot U_x(t)\setminus\O_x)\arrow[hook]{d}\arrow{u}[swap]{m_{c^{-1}}}\\
\pi^{-1}\big(U_x\big(\frac 1l\big)\setminus\O_x\big)\arrow{r}{\hat f_x\(\frac 1l\)^{-1}\!\!}&\big(\O_x\cap U_x(\frac 1l)\big)\times\pi^{-1}\big(\frac 1l*A\setminus\{x\}\big)\arrow{r}{v\times u_l}&\C^{d_x}\times\widetilde S_{x,0}\setminus\B_x\arrow{uuu}[swap]{\id\times m_{c^{-1}}}
\end{tikzcd}
\!\!\!\!\!}
\end{equation}

To prove this commutativity, let us introduce the notation $\hat f_x(1/l)^{-1}(z)=(q(z),\tilde s(z))$.
Clearly, we can choose $t$ so small that $U_x(t)$ and $c\cdot U_x(t)$ are both subsets of $f_x(1/l)(Q_c\times S_c)$,
where $Q_c$ and $S_c$ are the subsets given by Lemma~\ref{lemma:0} in the situation of Section~\ref{eq_nilpotent_cones}.

Takin preimages with respect to $\pi$ and applying part~\ref{part:i:corollary:0} of Corollary~\ref{corollary:0},
we get that $\pi^{-1}(U_x(t))$ and $\pi^{-1}(c\cdot U_x(t))$ are both contained in $\tilde f(1/l)(Q_c\times\widetilde S_c)$
where $\widetilde S_c=\pi^{-1}(S_c)$.
Subtracting $\pi^{-1}(\O_x)$, we obtain that
$\pi^{-1}(U_x(t)\setminus\O_x)$ and $\pi^{-1}(c\cdot U_x(t)\setminus\O_x)$ are both contained in
$\hat f(1/l)(Q_c\times\pi^{-1}(S_c\setminus\{x\}))$.

Part~\ref{part:ii:corollary:0} of Corollary~\ref{corollary:0} implies now that $q(c\cdot z)=c\cdot q(z)$ and
$\tilde s(c\cdot z)=c\cdot\tilde s(z)$ for any $z\in \pi^{-1}(U_x(t)\setminus\O_x)$ or $z\in\pi^{-1}(c\cdot U_x(t)\setminus\O_x)$.
The two maps corresponding to the upper ($p=0$) and lower ($p=1$) paths from $\pi^{-1}(c\cdot U_x(t)\setminus\O_x)$ to
the lower right corner in diagram~(\ref{eq:00}) are homotopic by
$$
F(z,p)\!:=\!\!\Big(pv(q(z))+(1-p)v(c^{-1}\cdot q(z)),pl\zeta(l*\tilde s(z))*c^{-1}\cdot\tilde s(z)+(1-p)l\zeta(l*c^{-1}\cdot\tilde s(z))*c^{-1}\cdot\tilde s(z)\Big),
$$
where $p\in[0,1]$. Note that here we took into account $*$-invariance~(\ref{eq:-4}). Passing to cohomologies in
diagram~(\ref{eq:00}), we obtain the commutative diagram
$$
\begin{tikzcd}[column sep=12pt]
&H^{2b_x-1}\big(\pi^{-1}\big(U_x\big(\frac1l\big){\setminus}\O_x\big),k\big)\arrow{ldd}[swap]{m_{c^{-1}}^*}\arrow{d}{\wr}&\arrow{l}[swap]{w^*\!\!\!\!}{\sim}H^{2b_x-1}\big(\C^{d_x}{\times}\widetilde S_{x,0}{\setminus}\B_x,k\big)\arrow{ddd}{(\id\times m_{c^{-1}})^*}\\
&H^{2b_x-1}\big(\pi^{-1}(U_x(t){\setminus}\O_x),k\big)\arrow{d}{m_{c^{-1}}^*}\\
H^{2b_x-1}\big(\pi^{-1}\big(c\hspace{1pt}{\cdot}\hspace{1pt}U_x\big(\frac1l\big){\setminus}\O_x\big),k\big)\arrow{r}{\sim}&H^{2b_x-1}\big(\pi^{-1}(c\hspace{1pt}{\cdot}\hspace{1pt}U_x(t){\setminus}\O_x),k\big)\\
&H^{2b_x-1}\big(\pi^{-1}\big(U_x\big(\frac 1l\big){\setminus}\O_x\big),k\big)\arrow{u}\arrow{ul}{\rho}&\arrow{l}[swap]{w^*\!\!\!\!}{\sim}H^{2b_x-1}\big(\C^{d_x}{\times}\widetilde S_{x,0}{\setminus}\B_x,k\big)
\end{tikzcd}\!\!\!
$$
where $w:=(v\times u_l)\circ\hat f_x\big(\frac 1l\big)^{-1}$ and all the unmarked arrows stem from the natural inclusions.
Here only the commutativity of the upper left triangle needs to be explained. Indeed, it follows from the following commutative diagram:
$$
\begin{tikzcd}
\pi^{-1}\big(U_x(\frac1l)\setminus\O_x\big)&\pi^{-1}\big(c\cdot U_x(\frac1l)\setminus\O_x\big)\arrow{l}[swap]{m_{c^{-1}}}\\
\pi^{-1}(U_x(t)\setminus\O_x\big)\arrow[hook]{u}&\pi^{-1}(c\cdot U_x(t)\setminus\O_x)\arrow[hook]{u}\arrow{l}[swap]{m_{c^{-1}}}
\end{tikzcd}
$$
We can conclude now that under all the identifications that we have made, the upper arrow in diagram~(\ref{eq:-1})
is just $(\id\times m_{c^{-1}})^*$ acting on $H^{2b_x-1}\big(\C^{d_x}\times\widetilde S_{x,0}\setminus\B_x,k\big)$.

Using the Poincar\'e duality, its functoriality with respect to orientation preserving isomorphisms\footnote{This fact
is a direct consequence of~\cite[Lemma 3.3.7]{Kashiwara_Schapira}.} (top and bottom rectangles)
and the K\"unneth formula, we obtain
$$
\begin{tikzcd}[column sep=13ex]
H^{2b_x-1}\big(\C^{d_x}\times\widetilde S_{x,0}\setminus\B_x,k\big)\arrow{r}{(\id\times m_{c^{-1}})^*}\arrow{d}{\wr}&H^{2b_x-1}\big(\C^{d_x}\times\widetilde S_{x,0}\setminus\B_x,k\big)\arrow{d}{\wr}\\
H_c^{2d_x+2b_x+1}(\C^{d_x}{\times}\widetilde S_{x,0}{\setminus}\B_x,k)^\vee\arrow{d}{\wr}&H_c^{2d_x+2b_x+1}(\C^{d_x}{\times}\widetilde S_{x,0}{\setminus}\B_x,k)^\vee\arrow{l}[swap]{((\id\times m_{c^{-1}})^*)^\vee\!\!\!\!}\arrow{d}{\wr}\\
H_c^{2b_x+1}(\widetilde S_{x,0}\setminus\B_x,k)^\vee\arrow{d}{\wr}&H_c^{2b_x+1}(\widetilde S_{x,0}\setminus\B_x,k)^\vee\arrow{d}{\wr}\arrow{l}{(m_{c^{-1}}^*)^\vee}\\
H^{2b_x-1}(\widetilde S_{x,0}\setminus\B_x,k)\arrow{r}{m_{c^{-1}}^*}&H^{2b_x-1}(\widetilde S_{x,0}\setminus\B_x,k)
\end{tikzcd}\!\!
$$
Thus our calculations, diagram~(\ref{eq:-1}) and the isomorphism $A(x)\cong C(x,y)/C(x,y)^0$ prove that
$\M_x\cong H^{2b_x-1}(\widetilde S_{x,0}\setminus\B_x,k)^\alpha$ as $A(x)$-modules in the notation of
Section~\ref{local_systems}. Hence we get
\begin{equation}\label{eq:I}
i^*\M\cong\Loc\big(H^{2b_x-1}(\widetilde S_{x,0}\setminus\B_x,k)\big)\cong\Loc\big(\rad\big(H^{2b_x}(\B_x,k)^\vee\big)\big).
\end{equation}

\end{document}